\documentclass[english,12pt,reqno,twoside]{amsart}
\usepackage[utf8]{inputenc}
\usepackage[T1]{fontenc}
\usepackage{amsmath}
\usepackage{amsthm}
\usepackage{amssymb}
\usepackage{amsfonts}

\usepackage{soul,color}
\usepackage{t1enc}
\usepackage{a4,indentfirst,latexsym,graphics,graphicx}
\usepackage{mathrsfs}
\usepackage{enumitem}
\usepackage[colorlinks=true,urlcolor=blue,
citecolor=red,linkcolor=blue,linktocpage,pdfpagelabels,
bookmarksnumbered,bookmarksopen]{hyperref}
\usepackage[english]{babel}
\usepackage[left=2.1cm,right=2.1cm,top=2.72cm,bottom=2.72cm]{geometry}

\usepackage[metapost]{mfpic}
\usepackage[hyperpageref]{backref}
\usepackage[colorinlistoftodos]{todonotes}
\usepackage{verbatim}

\definecolor{iceberg}{rgb}{0.44, 0.65, 0.82}
\newcounter{todocounter}

\hypersetup{
  linkbordercolor=cyan
}

\usepackage[normalem]{ulem}
\usepackage{braket}

\usepackage{mathtools} 
\mathtoolsset{showonlyrefs}

\usepackage{bbm}
\usepackage[noadjust]{cite}
\usepackage{url}
\usepackage{lmodern}

\usepackage[useregional]{datetime2}

\usepackage{fixmath}

\numberwithin{equation}{section}
\newtheorem{Th}{Theorem}[section]
\newtheorem{Prop}[Th]{Proposition}
\newtheorem{Lem}[Th]{Lemma}

\newtheorem{Def}[Th]{Definition}
\newtheorem{Rem}[Th]{Remark}

\newcommand{\C}{\mathbb{C}}
\newcommand{\R}{\mathbb{R}}

\newcommand{\N}{\mathbb{N}}

\newcommand{\cD}{{\mathcal D}}
\newcommand{\cE}{{\mathcal E}}

\newcommand{\cI}{{\mathcal I}}

\newcommand{\divv}{\mathrm{div}}

\newcommand{\loc}{\mathrm{loc}}

\newcommand{\rad}{\mathrm{rad}}

\newcommand{\triple}[1]{{\left\vert\kern-0.25ex\left\vert\kern-0.25ex\left\vert #1 
\right\vert\kern-0.25ex\right\vert\kern-0.25ex\right\vert}}

 \def\dd{\, {\rm d}}

\allowdisplaybreaks

\begin{document}

\title[Existence and regularity for a Grushin-Choquard equation]{Existence and regularity for an entire Grushin-Choquard equation}
 \author{Federico Bernini and Paolo Malanchini}

  \address[Federico Bernini]{\newline\indent Dipartimento di Matematica e Applicazioni,
          \newline\indent Università degli Studi di Milano-Bicocca,
          \newline\indent Via R. Cozzi, 55 - I-20125 Milan, Italy}
          \email{\href{mailto:federico.bernini@unimib.it}{federico.bernini@unimib.it}}
  \address[Paolo Malanchini]{\newline\indent Dipartimento di Matematica e Applicazioni,
          \newline\indent Università degli Studi di Milano-Bicocca,
          \newline\indent Via R. Cozzi, 55 - I-20125 Milan, Italy}
          \email{\href{mailto:p.malanchini@campus.unimib.it}{p.malanchini@campus.unimib.it}}



\begin{abstract}
    We consider the following Choquard equation
    $$
        -\Delta_\gamma u + u = \left(d(z)^{-\mu} \ast |u|^p\right)|u|^{p-2}u, \text{ in } \R^N,
    $$
    where $\Delta_\gamma$ is the Grushin operator. For a suitable range of the parameter $p$ we prove the existence of a mountain pass solution of the equation and we establish that the solution belongs to $L^q(\R^N)$ for all $q\in [2,\infty]$ and to $C^{0,\alpha}_\loc(\R^N)$ for some $\alpha \in (0,1)$. Additionally, we provide a Poho\v zaev type identity, which allows us to derive a nonexistence result for smooth solutions to our equation.
\end{abstract}

\maketitle
\begin{center}
\begin{minipage}{11cm}
    \tableofcontents
\end{minipage}
\end{center}

\section{Introduction}
The study of differential operators with degenerate coefficients has become an increasingly important research topic in mathematical analysis in recent years. This class of operators can be seen as a bridge between the classical theory of uniformly elliptic operators and hyperbolic operators, introducing intrinsically more complex geometries and functional structures. The most famous and extensively studied prototype is what is known today as the Grushin (or Baouendi-Grushin) operator, introduced by Baouendi \cite{baouendi} and Grushin \cite{grushin}. For the readers' convenience, we briefly recall its construction: if we split the Euclidean space $\R^N$ as $\R^m \times \R^\ell$, where $m \geq 1$, $\ell \geq 1$ satisfy $m+\ell=N$, a generic point $z \in \R^N$ can be written as
\begin{equation}
	z = \left(x,y\right) = \left( x_1,\ldots,x_m,y_1,\ldots,y_\ell \right),
\end{equation}
where $x \in \R^m$ and $y \in \R^\ell$. Given a nonnegative $\gamma \in \R$, the Grushin operator $\Delta_\gamma$ is defined  by
\begin{equation}
	\Delta_\gamma u(z) = \Delta_x u(z) + |x|^{2\gamma} \Delta_y u (z),
\end{equation} 
where $\Delta_x$ and $\Delta_y$ are the Laplace operators in the $x$ and $y$ variables, respectively. We also define the Grushin gradient 
\begin{equation}
	\nabla_\gamma u (z) =  (\nabla_x u(z), |x|^\gamma\nabla_y u(z))= (\partial_{x_1} u(z) , \dots, \partial_{x_m}u(z), |x|^\gamma \partial_{y_{1}}u(z), \dots, |x|^\gamma \partial_{y_{\ell}}u (z)).
\end{equation}
A crucial property of this operator is that it is \textit{not uniformly elliptic} in $\R^N$, since it is degenerate on the subspace $\{0\} \times \R^\ell$.
The number
$$
    N_\gamma \coloneqq m + (1+\gamma)\ell
$$
is called the \emph{homogeneous dimension} associated to the decomposition $N=m+\ell$.

One source of interest in Grushin-type operators arises from the fact that they have a strong connection with the Heisenberg sub-Laplacian if $\gamma=1$. Indeed, for $\gamma=1$, Jerison and Lee (see \cite{JL1,JL2}) studied the interplay between the equation
\begin{equation}
    \label{JL:eq}
    \Delta_\gamma u = -u^{\frac{N_\gamma+2}{N_\gamma-2}} \text{ in } \R^N,
\end{equation}
with the Yamabe problem, and, as observed in \cite{MM06}, the model space for \eqref{JL:eq} is the Heisenberg group $\C^N \times \R$. In this space, the Yamabe problem (for the Webster curvature) becomes
\begin{equation}
    \Delta_H u = -u^{\frac{N_\gamma+2}{N_\gamma-2}} \text{ in } \R^N,
\end{equation}
with $\Delta_H$ denoting the Heisenberg sub-Laplacian.

Moreover, the Grushin operator falls into the class of $X$-elliptic operators introduced in \cite{KL00}. In fact, the Grushin operator is uniformly $X$-elliptic with respect to the family of vector fields $X=(X_1,\dots, X_N)$ defined as
\begin{equation}
    X_i = \frac{\partial}{\partial x_i} \,\,\hbox{for $i=1,\ldots,m$} \qquad
    X_{m+j} = \vert x \vert^\gamma \frac{\partial}{\partial y_j} \,\,\hbox{for $j=1,\ldots,\ell$}.
\end{equation}
With this notation, we can write $\Delta_\gamma = \sum_{j=1}^{N} X_j^2$. The fact that the Grushin operator belongs to the class of $X$-elliptic operators will be important in the proof of one of the main results of the present work (see Theorem \ref{thm_reg}).
We point out that, if $\gamma\in\N$, the $C^\infty$-field $X=\left( X_1,\ldots,X_N \right)$ satisfies the H\"{o}rmander's finite rank condition, that is, see \cite{hormander1967}, the Lie algebra generated by the vector fields has constant dimension $N$.

\medskip 

Recently, growing attention has been devoted to the study of PDEs governed by such subelliptic operators. Most of the literature has mainly focused on linear aspects, such as the existence of the heat kernel (see \cite{GarofaloTralli}), asymptotic estimates (see \cite{L19}) and unique continuation properties for solutions (see \cite{abatangelo2024}).
On the other hand, from a variational perspective, the interest in these operators stems from the necessity to generalize and adapt fundamental analytical tools, e.g. such as Sobolev inequalities (see \cite{Loiudice2006}) and Hardy inequalities (see \cite{DAmbrosio2004,BBP}), to geometric contexts that lack classical Euclidean symmetries. These facts have led to a significant increase in interest in the study of nonlinear subelliptic PDEs involving Grushin-type operators: see the very recent contributions \cite{abatangelo2024,MalanchiniMolicaBisciSecchi2,MalanchiniMolicaBisciSecchi3,Loiudice26} and the references therein, where Grushin problems have been studied in different contexts. We would like to mention that the cited papers deal with equations in bounded domains: this, of course, is helpful in recovering compactness. Despite the deep understanding achieved for linear equations associated with the Grushin operator, the scenario changes radically when tackling nonlinear analysis on unbounded domains. To date, the literature on nonlinear Grushin-type PDEs in the whole space $\R^N$ is still rather limited; see \cite{AH23,AGLT24,BaldelliMalanchiniSecchi,MalanchiniMolicaBisciSecchi1}. Indeed, as in the classical case, recovering compactness typically requires arguments \`a la Lions combined with suitable group invariance (usually translation) of the equation. Unfortunately, this is not so obvious in the Grushin setting, and a classical scheme does not seem straightforward to implement. This explains why the literature on Grushin entire equations treated by variational methods is still rather limited.

\medskip

In this paper, we will focus on the following Grushin-Choquard equation
\begin{equation}
\label{GC}
    -\Delta_\gamma u + u = \left(d(z)^{-\mu} \ast |u|^p\right)|u|^{p-2}u, \quad  \text{ in } \R^N,
\end{equation}
where $N\ge 3, \mu>0, p>1$ and $d(z)$ is the Grushin distance defined as
\begin{equation}
    d(z)=(|x|^{2(\gamma+1)}+|y|^2)^\frac{1}{2(\gamma+1)}.
\end{equation}
Note that, when $\gamma = 0$, $\Delta_\gamma$ coincides with the classical Laplacian and $d(\cdot)$ is the Euclidean distance in $\R^N$ and so \eqref{GC} reduces to the classical Choquard equation, widely studied in the past decades by many authors. Given the vast literature on the subject, we refer to the important survey \cite{MVS2017} by Moroz and Van Schaftingen for a detailed description of the classical problem and the known results.

Equation \eqref{GC} is not new in the literature: indeed, in \cite{KanugoMishraMolicaBisci} the authors proved bifurcation and multiplicity results for a critical (in the sense of Hardy-Littlewood-Sobolev exponent) Grushin-Choquard equation in a bounded domain. Furthermore, it has also been studied in the whole space $\R^N$ in \cite{RL23} and the recent \cite{WCZH26}, where the authors investigate stability properties and Liouville-type theorems for a weighted Grushin–Choquard equation. 

We point out that these results address qualitative properties of solutions assuming their existence, whereas in the present paper we establish the existence of solutions by variational methods.

As already described, the main obstacle to applying classical variational methods for entire equations is primarily due to the difficulty in recovering global compactness in the whole space $\R^N$. Indeed, the search for critical points of the associated energy functional relies heavily on Sobolev embedding theorems. In the specific case of the Grushin operator, this lack of compactness is mainly due to two peculiar structural features: anisotropic dilations defined as
\begin{equation}
	\delta_t(x,y) = (t x, t^{\gamma+1}y),\quad t >0,
\end{equation}
and the breakdown of translation invariance. Unlike the classical Laplacian or the sub-Laplacian on the Heisenberg group, the Grushin operator is not invariant under translations along the $x$-directions, due to the presence of the spatial weight $|x|^{2\gamma}$. This structural asymmetry prevents the standard application of well-known tools used to recover compactness on unbounded domains, such as Lions' Concentration-Compactness Principle.

To overcome these limitations, the current literature often works within specific weighted Sobolev spaces or restricts the search to classes of functions endowed with particular symmetries (for instance, functions that are radially symmetric in $x$ and $y$), thereby restoring the lost compactness. In this paper, we follow the second approach.

\medskip

As anticipated, we argue variationally, so we associate to \eqref{GC} the energy functional
\begin{equation}
    \label{energy:functional}
    \cE_\gamma(u)
    \coloneqq
    \frac12\int_{\R^N}\left(|\nabla_\gamma u|^2 + |u|^2\right) \dd z 
    - 
    \frac{1}{2p}\int_{\R^N} (d(z)^{-\mu} \ast |u|^p)|u|^p \dd z.
\end{equation}
We set our problem in the Grushin-Sobolev space  $H^1_\gamma(\R^N)$  defined in \eqref{H1:grushin} below. Hence, at least formally, we can give the following definition.
\begin{Def}
\label{GC:weak:sol:def}
    We say that $u\in H^1_\gamma(\R^N)$ is a {\rm weak solution} to \eqref{GC} if 
    \begin{equation}
    \label{GC:weak:formulation}
        \cE_\gamma'(u)[v]
        =
        \int_{\R^N} \left(\nabla_\gamma u \cdot \nabla_\gamma v + uv\right) \dd z 
        - 
        \int_{\R^N} (d(z)^{-\mu} \ast |u|^p)|u|^{p-2}uv \dd z = 0,
    \end{equation}
    for every $v \in H^1_\gamma(\R^N)$.
\end{Def}

The main results of this work are the following. First, we will deal with an existence result: the restriction to the radial functions is crucial in the proof of this result in order to recover compactness (cf. Lemma \ref{lemma_compact}). Then we can obtain a critical point in the whole space $H^1_\gamma(\R^N)$ thanks to the \emph{Principle of Symmetric Criticality}.
\begin{Th} \label{thm_existence}
    Let $\mu\in (0,N_\gamma)$. If 
    \begin{equation}\label{eq_cond_p}
        \frac{N_\gamma-2}{2N_\gamma-\mu}< \frac{1}{p} < \frac{N_\gamma}{2N_\gamma-\mu},
    \end{equation}
 then there exists a weak solution $u\in H^1_\gamma(\R^N)$ to \eqref{GC}.
\end{Th} 
Once the existence of a weak solution is established, we investigate its regularity by using a Brezis-Kato and a bootstrap argument.
\begin{Th} \label{thm_reg}
    Let $\mu\in (0,4)$ and $p$ satisfies \eqref{eq_cond_p}. Let $u \in H^1_\gamma(\R^N)$ be a weak solution to \eqref{GC}.
    Then, $u\in L^q(\R^N)$ for all $q\in [2,+\infty]$. Moreover, $u\in C^{0,\alpha}_\loc(\R^N)$ for some $\alpha\in (0,1)$.
\end{Th}
Lastly, in the spirit of \cite[Proposition 3.1]{MVS2013}, we propose a Poho\v zaev identity for the Grushin-Choquard equation \eqref{GC} (see Proposition \ref{prop_Pohozaev}). This identity yields a nonexistence result, showing that the choice of $p$ in \eqref{eq_cond_p} is sharp, giving the analogous of \cite[Theorem 2]{MVS2013}. Some differences compared to the nondegenerate case, of course, occur. Indeed, in \cite[Proposition 3.1]{MVS2013}, the authors obtained the identity by testing the equation with a suitable function (see Section \ref{Sect:nex} for details). However, as observed in \cite{KL12} (see also \cite{L19}), in the Grushin setting we have to consider $C^2$-test functions. However, it is important to remark, that, at least to our knowledge, so far no results for regularity \textit{more than H\"older} are available, and, for the sake of accuracy, it seems nontrivial.

Our nonexistence result states as follows.
\begin{Th} \label{thm_nonexistence}
    Let $u\in C^2(\R^N)\cap H^1_\gamma(\R^N)$ be a weak solution to \eqref{GC}. If $\frac{1}{p}\le \frac{N_\gamma-2}{2N_\gamma-\mu}$ or $\frac{1}{p}\ge \frac{N_\gamma}{2N_\gamma-\mu}$ then $u\equiv 0$.
\end{Th}
A similar result in the Grushin setting has been proved by \cite{L19} for a Brezis-Nirenberg-type equation with Hardy potential in bounded domains and, recently, by \cite{YuYangTang} for elliptic equations with prescribed mass.

\medskip

\paragraph{\textbf{Overview}} The paper has the following outline. In Section \ref{Sect:preliminaries} we will set the functional space in which we consider our problem, stating some important embedding results. Then, thanks to the Hardy-Littlewood-Sobolev inequality adapted to the Grushin setting (cf. Lemma \ref{lemma_HLS}), we can prove that the associated functional is of class $C^1$ (see Proposition \ref{funct:reg}).

In Section \ref{Sect:ex} we prove the existence of a mountain pass solution for our problem: we will begin by showing that the functional has the right geometry, and then we provide an analysis on the Palais-Smale sequences, proving that they are bounded and, eventually, that the $(PS)$-condition at the mountain-pass level is satisfied. 

In Section \ref{Sect:reg} we prove the desired regularity for weak solutions to \eqref{GC}.

Finally, in Section \ref{Sect:nex} we first prove of the Poho\v zaev identity and then we provide the nonexistence result.

\section{Preliminaries}
\label{Sect:preliminaries}
This section contains all the preliminary results needed in the rest of the paper and is divided into two parts: in the first one, we present the functional setting in which the problem is considered, defining the space, and recalling some important embedding properties. In the second part, we report the generalization to Grushin spaces of the Hardy-Littlewood-Sobolev, together with the proof of the regularity of our energy functional \eqref{energy:functional}.
\subsection{Functional setting}
Given $q\in[1,+\infty]$, $L^q(\R^N)$ stands for the standard Lebesgue space, whose norm will be indicated by $\|\cdot\|_q$. 

We define the inner product
\begin{equation*}
    (u,v)_\gamma = \int_{\R^N} \left( \nabla_\gamma u \cdot \nabla_\gamma v + uv\right)\dd z,
\end{equation*}
which implies the norm
\begin{equation}
    \label{Grushin:norm}
    \|u\|_\gamma = \left(\int_{\R^N} (|\nabla_\gamma u|^2 + |u|^2 ) \dd z\right)^\frac12,
\end{equation}
and so we consider the Grushin-Sobolev space
\begin{equation}
    \label{H1:grushin}
    H^1_\gamma(\R^N) = \left\{u \in L^2(\R^N) : \|u\|_\gamma < +\infty \right\},
\end{equation}
which is clearly a Hilbert space.
In \cite[Lemma 5.1]{Loiudice2006}, the author showed that 
\begin{equation}
    \|u\|_{2^*_\gamma}^2\leq C\int_{\R^N}|\nabla_\gamma u|^2\dd z\le C \int_{\R^N} \left(|\nabla_\gamma u|^2 + |u|^2 \right)\dd z = C\|u\|_\gamma^2
\end{equation}
for all $u \in C^\infty_c(\R^N)$, where $2^*_\gamma=2N_\gamma/(N_\gamma-2)$ is the critical Sobolev-Grushin exponent. Then, arguing similarly as  in \cite{Loiudice2006}, this implies that the function $u \mapsto \|u\|_\gamma$ is a norm on $C^\infty_c(\R^N)$. Hence, if we set $H^1_{0,\gamma}(\R^N)$ as the completion of $C^\infty_c(\R^N)$ with respect to the norm \eqref{Grushin:norm}, then $H^1_{0,\gamma}(\R^N)$ is a Hilbert space, with the same norm of \eqref{H1:grushin}: therefore, they must coincide. Furthermore, since, as observed in \cite[p. 1249]{AH23}, it holds
\begin{equation}
    \|u\|_{2^*_\gamma}\leq C\|u\|_\gamma,
\end{equation}
for every $u \in H^1_\gamma(\R^N)$, then by interpolation the embedding
\begin{equation}
    \label{H1gamma:emb}
    H^1_\gamma(\R^N) \hookrightarrow L^q(\R^N)
\end{equation}
is continuous for every $q \in [2,2^*_\gamma]$.

We also define the space of radial functions
\begin{equation}
\label{GS:space:radial}
    H^1_{\gamma,\rad_{x,y} }(\R^N)\coloneqq \{ u\in H^1_\gamma(\R^N)\, : u(x,y) = u(|x|, |y|), \,\hbox{for all $(x,y)\in\R^N$}\}
\end{equation}
and recall the compact embedding
\begin{Lem}[{\cite[Lemma 2.3]{AH23}}]\label{lemma_compact}
Let $q\in (2,2^*_\gamma)$. Then the embedding $H^1_{\gamma,\rad_{x,y} }(\R^N)\hookrightarrow L^q(\R^N)$ is compact.
\end{Lem}

\subsection{Some tools}
The following result is a generalization to the Grushin setting of the Hardy-Littlewood-Sobolev inequality, due to \cite[Proposition 2.2]{KanugoMishraMolicaBisci}.
\begin{Lem}
\label{lemma_HLS}
    Let $r,s>1$ and $0 < \mu < N_\gamma$, with $\frac{1}{r} + \frac{\mu}{N_\gamma}+\frac1s = 2$. Let $f \in L^r(\R^N)$ and $h \in L^s(\R^N)$. There exists a sharp constant $C\coloneqq C(N_\gamma,\mu,r,s)$, independent of $f,h$ such that
    \begin{equation}
        \label{Grushin:HLS}
        \int_{\R^N \times \R^N} \frac{f(z)h(w)}{d(z-w)^\mu} \dd z \dd w \leq C\|f\|_r\|h\|_s.
    \end{equation}
\end{Lem}
Thanks to the previous Lemma we can now prove the regularity of the energy functional $\cE_\gamma$, defined in \eqref{energy:functional}. For the sake of simplicity, we denote the convolution as
\begin{equation}
    \cD(u) = \int_{\R^N} (d(z)^{-\mu} \ast |u|^p)|u|^p \dd z
\end{equation}
and we will use it indiscriminately in both ways, depending on the situation. We have
\begin{Prop}
\label{funct:reg}
    Let $p >1$. Then the energy functional $\cE_\gamma$ is of class $C^1$ on $H^1_\gamma(\R^N)\cap L^{\frac{2pN_\gamma}{2N_\gamma-\mu}}(\R^N)$.
\end{Prop}
\begin{proof}
    By Lemma \ref{lemma_HLS} there exists a positive constant $C$ such that
    \begin{equation}
    \label{HLS:control:on:D}
        \int_{\R^N} (d(z)^{-\mu} \ast |u|^p)|u|^p \dd z 
        = 
        \int_{\R^N \times \R^N} \frac{|u(z)|^p|u(w)|^p}{d(z-w)^\mu} \dd z \dd w \leq 
        C \||u|^p\|_{{\frac{2N_\gamma}{2N_\gamma-\mu}}}^2 = C\|u\|_{{\frac{2pN_\gamma}{2N_\gamma-\mu}}}^{2p}.
    \end{equation}
Hence,
\begin{align}
    |\cE_\gamma(u)| 
    \leq 
    \frac12\|u\|_\gamma^2 + \int_{\R^N} (d(z)^{-\mu} \ast |u|^p)|u|^p \dd z
    \leq
    \frac12\|u\|_\gamma^2 + C\|u\|_{{\frac{2pN_\gamma}{2N_\gamma-\mu}}}^{2p},
\end{align}
so $\cE_\gamma$ is well defined. Now, let $(u_n) \subset H^1_\gamma(\R^N) \cap L^{\frac{2pN_\gamma}{2N_\gamma-\mu}}(\R^N)$ be a sequence such that $u_n \to u$ in $H^1_\gamma(\R^N) \cap L^{\frac{2pN_\gamma}{2N_\gamma-\mu}}(\R^N)$, and we compute
\begin{align}
    |\cE_\gamma(u_n) - \cE_\gamma (u)|
    &\leq
    \frac12\|u_n\|_\gamma^2-\frac12\|u\|_\gamma^2 +\frac{1}{2p}(\cD(u_n)-\cD(u))\\
    & = 
    \frac12\|u_n\|_\gamma^2-\frac12\|u\|_\gamma^2 \\
    &\quad +\frac{1}{2p}\int_{\R^N \times \R^N} \frac{|u_n(z)|^p(|u_n(w)|^p - |u(w)|^p) + |u(w)|^p(|u_n(z)|^p - |u(z)|^p}{d(z-w)^\mu} \dd z \dd w 
\end{align}
that goes to $0$ as $n\to\infty$ thanks to the weakly lower semicontinuity of the norm and \eqref{HLS:control:on:D}.
Hence, the energy functional $\cE_\gamma$ is continuous.\\
Computing the G\^{a}teuax derivative of $\cE_\gamma$ on $u \in H^1_\gamma(\R^N) \cap L^{\frac{2pN_\gamma}{2N_\gamma-\mu}}(\R^N)$ along $v \in H^1_\gamma(\R^N) \cap L^{\frac{2pN_\gamma}{2N_\gamma-\mu}}(\R^N)$, we get
\begin{equation}
    \cE'_\gamma(u)[v]=(u,v)_\gamma + \int_{\R^N} (d(z)^{-\mu} \ast |u|^p)|u|^{p-2}uv \dd z.
\end{equation}
By very similar computations, we can also show that $\cE_\gamma'$ is continuous, getting the desired result.
\end{proof}

\begin{Rem}
Note that the condition \eqref{eq_cond_p} ensures that $\frac{2pN_\gamma}{2N_\gamma-\mu}\in (2,2^*_\gamma)$ and so by the Sobolev embedding in \eqref{H1gamma:emb}, $H^1_\gamma(\R^N)\hookrightarrow L^{\frac{2pN_\gamma}{2N_\gamma-\mu}}(\R^N)$. Therefore, under this restriction on $p$, the functional $\cE_\gamma$ is well defined and of class $C^1$ on $H^1_\gamma(\R^N)$.
\end{Rem}

\section{Existence}
\label{Sect:ex}
In this section, we produce a nontrivial solution to \eqref{GC}, by means of the Mountain Pass Theorem, see \cite[Theorem 2.1]{ambrosettirabinowitz}. 
We recall that we are working in the subspace of the radial functions \eqref{GS:space:radial}, so we are considering the functional $\cE_{\gamma,\rad}:H^1_{\gamma,\rad_{x,y}}(\R^N) \to \R$. For the sake of simplicity, we will continue to denote the restricted functional with $\cE_\gamma$.
We begin by showing that $\cE_\gamma$ satisfies the mountain pass geometry. 

\begin{Lem}\label{lemma_PS1}
Let $p$ satisfy \eqref{eq_cond_p} and let $u \in H^1_{\gamma,\rad_{x,y}}(\R^N) $. Then the following holds.
    \begin{enumerate}[label=(\roman*)]
		\item There exist $R>0$ and $a>0$ such that if $\|u\|_\gamma=R$, then $\cE_\gamma(u)\ge a$. 
		\item There exists $v\in H^1_{\gamma,\rad_{x,y}}(\R^N)$ such that $\|v\|_\gamma>R$ and $\cE_\gamma(v)<0$.
    \end{enumerate}

\end{Lem}
\begin{proof}
    We begin by proving $(i)$. Thanks to the Lemma \ref{lemma_HLS}, if $u\in H^1_{\gamma,\rad_{x,y}}(\R^N)$ then there exists $C_1>0$ such that $\cD(u)\le C_1\|u\|_{\frac{2pN_\gamma}{2N_\gamma-\mu}}^{2p}$. Moreover, by the embedding $H^1_{\gamma,\rad_{x,y}}(\R^N)\hookrightarrow L^{\frac{2pN_\gamma}{2N_\gamma-\mu}}(\R^N)$ in Lemma \ref{lemma_compact} there exists $C_2>0$ such that
        $
        \|u\|_{\frac{2pN_\gamma}{2N_\gamma-\mu}}\le C_2 \|u\|_\gamma
        $
        and so 
        \begin{equation}
            \cD(u)\le C \|u\|_\gamma^{2p},
        \end{equation}
        for some $C>0$. That is
        $$
            \cE_\gamma(u)= \frac12 \|u\|_\gamma^2 - \frac{1}{2p} \cD(u)\ge \frac{1}{2}\|u\|_\gamma^2 - \frac{C}{2p}\|u\|_\gamma^{2p}.
        $$
        Now, if $\|u\|_\gamma = R$,
        $$
            \cE_\gamma(u)\ge \frac{1}{2}R^2 - \frac{C}{2p}R^{2p}>0
        $$
        for $R$ sufficiently small, since $p>1$.\\
        Concerning  $(ii)$, let $\varphi\in C^\infty_c(\R^N)\setminus\{0\}$ such that $\|\varphi\|_\gamma =1$. Given $t>0$ we evaluate
        $$
            \cE_\gamma(t\varphi) = \frac{t^2}{2} - \frac{1}{2p}\cD(t\varphi) = \frac{t^2}{2}- \frac{t^{2p}}{2p}\cD(\varphi).
        $$
        Since $\cD(\varphi)>0$, there exists $t_1>0$ such that $\cE_\gamma(t\varphi)<0$ for all $t>t_1,$ concluding the proof.
\end{proof}
The next step is to verify the validity of the Palais-Smale condition.
\begin{Lem}\label{lemma_PS2}
    Let $(u_n)\subset H^1_{\gamma,\rad_{x,y}}(\R^N)$ be a Palais-Smale sequence at level $c>0$, that is 
    \begin{equation}\label{eq_PS}
        \cE_\gamma(u_n)\to c \quad  \cE'(u_n)\to 0
    \end{equation}
    as $n\to\infty$.
    The functional $\cE_\gamma$ satisfies the Palais-Smale condition.
\end{Lem}
\begin{proof}
    Let $(u_n)$ be as in the assumption. Our goal is to show that, up to a subsequence, $(u_n)$ converges strongly to a nontrivial limit point $u \in H^1_{\gamma,\rad_{x,y}}(\R^N)$.\\
    First, we prove that $(u_n)$ is bounded in $H^1_{\gamma,\rad_{x,y}}(\R^N)$. By \eqref{eq_PS} there exists $n_0>0$ such that
    \begin{equation}
        |\cE_\gamma'(u_n)[ u_n]|\le \|u_n\|_\gamma
    \end{equation}
    for all $n\ge n_0$ and there exists $M>0$ such that
    $$
        |\cE_\gamma(u_n)|\le M
    $$
    for all $n\in \N$.
    So we can estimate
    \begin{equation}
        M+\frac{1}{2p}\|u_n\|_\gamma \le \cE_\gamma(u_n) - \frac{1}{2p} \cE_\gamma'(u_n)[u_n] = \left( \frac12 -\frac{1}{2p}\right) \|u_n\|_\gamma^2
    \end{equation}
for all $n\ge n_0$, which implies that $\|u_n\|_\gamma$ is bounded.\\
So, up to subsequence, there exists $u\in H^1_{\gamma,\rad_{x,y}}(\R^N)$ such that $u_n\rightharpoonup u$ in $H^1_{\gamma,\rad_{x,y}}(\R^N)$ and 
    \begin{equation}\label{eq_strong}
        u_n\to u \quad\hbox{strongly in $L^q(\R^N)$ for all $q\in (2,2^*_\gamma)$},
    \end{equation}
by the compact embedding in Lemma \ref{lemma_compact}.\\ 
Now we want to prove that $u_n\to u$ strongly in $H^1_{\gamma,\rad_{x,y}}(\R^N)$. Note that
\begin{equation}
\label{eq_ps_conv}
\begin{aligned}
    0\leftarrow  \cE_\gamma'(u_n)[u_n -u] 
    &= \int_{\R^N}\nabla_\gamma u_n \nabla_\gamma(u_n-u)\dd z + \int_{\R^N} u_n (u_n-u)\dd z \\
    &\quad + \int_{\R^N\times \R^N} \frac{|u_n(w)|^p}{d(z-w)^\mu} |u_n(z)|^{p-2} u_n(z) \left( u_n(z)-u(z)  \right)\dd z\dd w,
\end{aligned}
\end{equation}
as $n\to\infty$.
Now, by Lemma \ref{lemma_HLS} with $r=s = 2N_\gamma/(2N_\gamma-\mu)$,
$$
\int_{\R^N\times \R^N} \frac{|u_n(w)|^p}{d(z-w)^\mu} |u_n(z)|^{p-2} u_n(z) \left( u_n(z)-u(z)  \right)\dd z\dd w \le C\|u_n\|_{\frac{2pN_\gamma}{2N_\gamma-\mu}}^p \|f\|_r,
$$
where $f\coloneqq|u_n|^{p-2} u_n (u_n -u)$. By H\"older's inequality with $p$ and $p' = p/(p-1)$
$$
\|f\|_r^r = \int_{\R^N} |u_n|^{r(p-1)}|u_n-u|^r \dd z \le \|u_n\|_{rp}^{r(p-1)}\|u_n-u\|_{rp}^r.
$$
So
$$
\int_{\R^N\times \R^N} \frac{|u_n(w)|^p}{d(z-w)^\mu} |u_n(z)|^{p-2} u_n(z) \left( u_n(z)-u(z)  \right)\dd z\dd w \le C \|u_n\|_{\frac{2pN_\gamma}{2N_\gamma-\mu}}^{2p-1}\|u_n-u\|_{\frac{2pN_\gamma}{2N_\gamma-\mu}}\to 0
$$
as $n\to\infty$ since $u_n\to u$ strongly in $L^{ \frac{2pN_\gamma}{2N_\gamma-\mu} }(\R^N)$, as observed in \eqref{eq_strong}. By \eqref{eq_ps_conv} follows that
\begin{equation}\label{eq_PSc1}
    \int_{\R^N}\nabla_\gamma u_n \cdot \nabla_\gamma(u_n-u)\dd z +\int_{\R^N} u_n(u_n-u)\dd z\to 0
\end{equation}
as $n\to\infty$. Moreover, by the weak convergence of $u_n\rightharpoonup u$ in $H^1_{\gamma,\rad_{x,y}}(\R^N)$ we also have that
\begin{equation}\label{eq_PSc2}
    \int_{\R^N} \nabla_\gamma u \cdot \nabla_\gamma (u_n-u)\dd z + \int_{\R^N} u (u_n-u)\dd z\to 0
\end{equation}
as $n\to \infty$. Combining \eqref{eq_PSc1} and \eqref{eq_PSc2} we get that
\begin{equation}
    \|u_n-u\|_\gamma^2\to 0,
\end{equation}
that is,
\begin{equation}
    u_n \to u \text{ strongly in } H^1_{\gamma,\rad_{x,y}}(\R^N).
\end{equation}
To conclude, we just need to verify that $u$ is a nontrivial limit point. 
Suppose by contradiction that $u=0$. By \eqref{eq_strong} $u_n \to 0$ in $L^q(\R^N)$ for every $q \in (2,2^*_\gamma)$ and by \eqref{HLS:control:on:D},
\begin{equation}
    \cD(u_n)= \int_{\R^N} (d(z)^{-\mu} \ast |u_n|^p)|u_n|^p \dd z \leq C \|u_n\|^p_{\frac{2pN_\gamma}{2N_\gamma - \mu}} \to 0, \quad \text{ as } n \to \infty.
\end{equation}
So,
\begin{equation}
    \cE_\gamma(u_n) = \frac12\|u_n\|_\gamma^2- \frac{1}{2p}\cD(u_n)  \to 0, 
\end{equation}
as $n\to\infty$, which is a contradiction with \eqref{eq_PS}, hence $u\ne 0$.
\end{proof}

We are now in position to prove Theorem \ref{thm_existence}.
\begin{proof}[Proof of Theorem \ref{thm_existence}]
By Lemma \ref{lemma_PS1} and \ref{lemma_PS2}, the hypotheses of the Mountain Pass Theorem \cite[Theorem 2.1]{ambrosettirabinowitz} are verified. So, there exists a nontrivial critical point $u\in H^1_{\gamma,\rad_{x,y}}(\R^N)$ of $\cE_\gamma$ satisfying
 \begin{equation}
 	\label{eq:13}
 	\cE_\gamma(u) =\inf_{\sigma\in\Sigma_v} 
 	 \max_{t\in [0,1]} \cE_\gamma(\sigma(t))\ge a ~\hbox{and $\cE'_\gamma(u) =0$},
 \end{equation} 
where $\Sigma_v = \{\sigma\in C\left([0,1], H^1_{\gamma,\rad_{x,y}}(\R^N)\right):\, \sigma(0) = 0, \sigma(1) = v\}$, where $a$ and $v$ are as in Lemma \ref{lemma_PS1}.\\
We now show that $u$ is in fact a critical point of $\cE_\gamma$ over the entire space $H^1_\gamma(\R^N)$ by the \emph{Principle of Symmetric Criticality} \cite[Theorem 1.28]{willem}.  Following the argument in \cite[p.~1260]{AH23} (see also \cite{AGLT24}), we denote by
\begin{equation}
    G\coloneqq \{g\in O(N)\, : g(x,y) = \left( h(x), r(y) \right)\text{with $h\in O(m)$ and $r\in O(\ell)$}  \}
\end{equation}
and the action $G\times H^1_\gamma(\R^N)\times H^1_\gamma(\R^N)$ given by
\begin{equation}
    (gu)(x,y) = u(h(x), r(y)) \quad \text{ for all } (x,y)\in \R^m\times \R^\ell,
\end{equation}
which is isometric, that is, $\|gu\|_\gamma = \|u\|_\gamma$. Finally, since the Grushin distance satisfies
\begin{equation}
    d(g(z), g(w)) = d(z,w) \text{ for all } z,w\in \R^N
\end{equation}
and $g\in G$ we have $\cE_\gamma(gu) = \cE_\gamma(u)$. And so, by \cite[Theorem 1.28]{willem}, $u$ is a critical point of $\cE_\gamma$ over $H^1_\gamma(\R^N)$.

\end{proof}

\section{Regularity}
\label{Sect:reg}
In this section we investigate the regularity of weak solutions of \eqref{GC}. The proof relies on a Brezis-Kato argument (i.e. by a truncation argument), and then by De Giorgi's iteration technique. We have been inspired by the proof given in \cite{Ambrosio}, where the author deals with the regularity of weak solutions for a quasilinear Choquard equation. 

For the sake of readability, we split the proof of the regularity result (Theorem~\ref{thm_reg}) into two lemmas. Throughout this section, we assume the hypotheses of Theorem~\ref{thm_reg}; namely, $\mu \in (0,4)$, and $p$ satisfies \eqref{eq_cond_p}.

We observe that in the case of the Laplace operator (see, for instance \cite{MVS2013}), the stronger assumption on $\mu$ is not needed. In those works, the authors establish the regularity of solutions relying on the classical Calderón–Zygmund $L^q$-regularity theory. However, in the setting of the Grushin operator, the validity of such estimates seems not trivial to obtain, which prevents us from following the same approach.

\begin{Lem} \label{lem_reg_1}
Let $u \in H^1_\gamma(\R^N)$ be a solution to \eqref{GC}. Then $u \in L^q(\R^N)$ for all $q\in [2,\infty)$. 
\end{Lem}
\begin{proof}
Without loss of generality, we can assume $u \geq 0$. Otherwise, we can separately repeat the computations for $u^+=\max\{u,0\}$ and $u^-=\max\{-u,0\}$. Let $M > 1$ and setting $u_M=\min\{u,M\}$, we consider $u_M^{2k+1},\, k \geq 0$, as a test function in \eqref{GC:weak:formulation}, that is
    \begin{equation}
        \int_{\R^N} \left(\nabla_\gamma u \cdot \nabla_\gamma u_M^{2k+1} + uu_M^{2k+1}\right) \dd z 
        - 
        \int_{\R^N} (d(z)^{-\mu} \ast |u|^p)|u|^{p-2}uu_M^{2k+1} \dd z.
    \end{equation}
    We first observe that $\nabla_\gamma u_M = \nabla_\gamma u$, by definition, and $\nabla_\gamma (u_M^{2k+1})=(2k+1)u^{2k}\nabla_\gamma(u_M)$: hence,
    \begin{equation}
    \label{nabla:gamma:int:1}
        \int_{\R^N} \nabla_\gamma u \cdot \nabla_\gamma(u_M^{2k+1}) \dd z = (2k+1)\int_{\R^N}u_M^{2k}|\nabla_\gamma u_M|^2 \dd z.
    \end{equation}
    Now, since $\nabla_\gamma(u_M^{k+1})=(k+1)u_M^k\nabla_\gamma (u_M)$, it follows that $|\nabla_\gamma(u_M^{k+1})|^2=(k+1)^2u_m^{2k}|\nabla_\gamma (u_M)|^2$, that is
    \begin{equation}
        \label{nabla:gamma:int:2}
        \int_{\R^N}u_M^{2k}|\nabla_\gamma(u_M)|^2 \dd z = \frac{1}{(k+1)^2}\int_{\R^N}|\nabla_\gamma(u_M^{k+1})|^2 \dd z.
    \end{equation}
    Therefore, by \eqref{nabla:gamma:int:1} and \eqref{nabla:gamma:int:2} we get
    \begin{equation}
    \label{first:integral:control}
        \int_{\R^N} \nabla_\gamma u \cdot \nabla_\gamma(u_M^{2k+1}) \dd z
        =
        \frac{(2k+1)}{(k+1)^2}\int_{\R^N}|\nabla_\gamma(u_M^{k+1})|^2 \dd z
        \geq
        S_\gamma\frac{(2k+1)}{(k+1)^2}\left(\int_{\R^N}|u_M^{k+1}|^{2^*_\gamma}\dd z\right)^{\frac{2}{2^*_\gamma}},
    \end{equation}
    where in the last inequality we used the Sobolev inequality for Grushin spaces (see Section \ref{Sect:preliminaries}).
    Again by definition, $u_M \leq u$, hence
    \begin{equation}
    \label{second:integral:control}
        \int_{\R^N}uu_M^{2k+1} \dd z
        \geq
        \int_{\R^N}u_Mu_M^{2k+1} \dd z
        =
        \int_{\R^N}|u_M^{k+1}|^2 \dd z.
    \end{equation}
    We deal now with the convolutive term. By Lemma \ref{lemma_HLS} (with $r=s$, hence $r=2N_\gamma/(2N_\gamma-\mu)$),
    \begin{equation}
    \label{third:integral:control}
        \begin{aligned}
            \left|\int_{\R^N}(d(z)^{-\mu} * |u|^p)|u|^{p-2}uu_M^{2k+1} \dd z \right|
            &\leq
            C\left[\int_{\R^N}|u|^{\frac{2pN_\gamma}{2N_\gamma-\mu}} \dd z
            \int_{\R^N}(|u|^{p-1}|u_M|^{2k+1})^{\frac{2N_\gamma}{2N_\gamma-\mu}} \dd z \right]^{\frac{2N_\gamma-\mu}{2N_\gamma}}\\
            & \le
            C \|u\|_{\frac{2pN_\gamma}{2N_\gamma-\mu}}^p
            \left(\int_{\R^N}|u|^{(2k+p)\frac{2N_\gamma}{2N_\gamma-\mu}} \dd z \right)^{\frac{2N_\gamma-\mu}{2N_\gamma}}\\
            & \le
            C 
            \left(\int_{\R^N}|u|^{(2k+p)\frac{2N_\gamma}{2N_\gamma-\mu}} \dd z \right)^{\frac{2N_\gamma-\mu}{2N_\gamma}}.
        \end{aligned}
    \end{equation}
    Hence, from \eqref{first:integral:control}, \eqref{second:integral:control}, and \eqref{third:integral:control} we obtain
    \begin{equation}
        \begin{aligned}
            S_\gamma\frac{(2k+1)}{(k+1)^2}\left(\int_{\R^N}|u_M^{k+1}|^{2^*_\gamma}\dd z\right)^{\frac{2}{2^*_\gamma}}
            &\leq
            S_\gamma\frac{(2k+1)}{(k+1)^2}\left(\int_{\R^N}|u_M^{k+1}|^{2^*_\gamma}\dd z\right)^{\frac{2}{2^*_\gamma}} + \int_{\R^N} |u_M^{k+1}|^2 \dd z\\
            &\leq
            C
            \left(\int_{\R^N}|u|^{(2k+p)\frac{2N_\gamma}{2N_\gamma-\mu}} \dd z \right)^{\frac{2N_\gamma-\mu}{2N_\gamma}}.
        \end{aligned}
    \end{equation}
    By the dominated convergence theorem, we can send $M \to +\infty$, obtaining
    \begin{equation}
    \label{estimate:after:limit}
        S_\gamma\frac{(2k+1)}{(k+1)^2}\left(\int_{\R^N}|u^{k+1}|^{2^*_\gamma}\dd z\right)^{\frac{2}{2^*_\gamma}}
        \leq
        C 
        \left(\int_{\R^N}|u|^{(2k+p)\frac{2N_\gamma}{2N_\gamma-\mu}} \dd z \right)^{\frac{2N_\gamma-\mu}{2N_\gamma}}.
    \end{equation}
    At this point, the idea is to work with the integral to the right-hand side to let appear the same integral on the left-hand side times a suitable multiplicative constant (i.e. less than one) in order to absorb it. To this end, let $\delta >1$, and we split the second integral into the sum where $|u| \leq \delta$ and $|u| > \delta$.\\
    Let us start with the ``far-away'' integral: by H\"{o}lder inequality (with $s=(2^*_\gamma(2N_\gamma-\mu)/(4N_\gamma)$ and $s'=(2^*_\gamma(2N_\gamma-\mu)/(2^*_\gamma(2N_\gamma-\mu)-4N_\gamma) = (2N_\gamma-\mu)/(4-\mu)$ ), because $|u| > \delta > 1$, and remembering the range of $p$, we have
    \begin{equation}
    \label{integral:far}
        \begin{aligned}
            \left(\int_{|u| > \delta}|u|^{(2k+p)\frac{2N_\gamma}{2N_\gamma-\mu}} \dd z \right)^{\frac{2N_\gamma-\mu}{2N_\gamma}}
            & =
            \left(\int_{|u| > \delta}|u|^{(p-2)\frac{2N_\gamma}{2N_\gamma-\mu}} |u|^{2(k+1)\frac{2N_\gamma}{2N_\gamma-\mu}} \dd z \right)^{\frac{2N_\gamma-\mu}{2N_\gamma}}\\
            & \leq
            \left(\int_{|u| > \delta}|u|^{\left(\frac{2N_\gamma-\mu}{N_\gamma-2}-2\right)\frac{2N_\gamma}{2N_\gamma-\mu}} |u|^{2(k+1)\frac{2N_\gamma}{2N_\gamma-\mu}} \dd z \right)^{\frac{2N_\gamma-\mu}{2N_\gamma}}\\
            & \leq
            \left(\int_{|u| > \delta}|u|^{2^*_\gamma} \dd z \right)^{\frac{4-\mu}{2N_\gamma}}
            \left(\int_{|u| > \delta}|u|^{(k+1)2^*_\gamma} \dd z \right)^{\frac{2}{2^*_\gamma}}\\
            & =:
            D(\delta)
            \left(\int_{|u| > \delta}|u|^{(k+1)2^*_\gamma} \dd z \right)^{\frac{2}{2^*_\gamma}},
        \end{aligned}
    \end{equation}
    where $D(\delta) \to 0$ as $\delta \to +\infty$, since $\mu<4$. We estimate now the ``close'' integral, hence
    \begin{equation}
    \label{integral:close}
        \begin{aligned}
            \left(\int_{|u| \leq \delta}|u|^{(2k+p)\frac{2N_\gamma}{2N_\gamma-\mu}} \dd z \right)^{\frac{2N_\gamma-\mu}{2N_\gamma}}
            & =
            \left(\int_{|u| \leq \delta}|u|^{(p-2)\frac{2N_\gamma}{2N_\gamma-\mu}} |u|^{2(k+1)\frac{2N_\gamma}{2N_\gamma-\mu}} \dd z \right)^{\frac{2N_\gamma-\mu}{2N_\gamma}}\\
            & \leq
            \delta^{p-2}\left(\int_{|u| \leq \delta} |u|^{(k+1)\frac{4N_\gamma}{2N_\gamma-\mu}} \dd z \right)^{\frac{2N_\gamma-\mu}{2N_\gamma}}.
        \end{aligned}
    \end{equation}
    Plugging \eqref{integral:far} and \eqref{integral:close} in \eqref{estimate:after:limit}
    \begin{equation}
        \begin{aligned}
            &S_\gamma\frac{(2k+1)}{(k+1)^2}\left(\int_{\R^N}|u^{k+1}|^{2^*_\gamma}\dd z\right)^{\frac{2}{2^*_\gamma}}\\
            &\quad \leq
            C
            \left[
            \delta^{p-2}\left(\int_{|u| \leq \delta} |u|^{(k+1)\frac{4N_\gamma}{2N_\gamma-\mu}} \dd z \right)^{\frac{2N_\gamma-\mu}{2N_\gamma}}
            +
            D(\delta)
            \left(\int_{|u| > \delta}|u|^{(k+1)2^*_\gamma} \dd z \right)^{\frac{2}{2^*_\gamma}}
            \right]\\
            & \quad \leq   
            C
            \left[
            \delta^{p-2}\left(\int_{\R^N} |u|^{(k+1)\frac{4N_\gamma}{2N_\gamma-\mu}} \dd z \right)^{\frac{2N_\gamma-\mu}{2N_\gamma}}
            +
            D(\delta)
            \left(\int_{\R^N}|u|^{(k+1)2^*_\gamma} \dd z \right)^{\frac{2}{2^*_\gamma}}
            \right].
        \end{aligned}
    \end{equation}
    Now we can choose $\delta>1$ large enough such that
    \begin{equation}
        0 < D(\delta) < \frac{\theta S_\gamma}{C}\frac{(2k+1)}{(k+1)^2}, \quad \text{for some}~\theta \in (0,1),
    \end{equation}
    hence
    \begin{equation}
            (1-\theta)S_\gamma\frac{(2k+1)}{(k+1)^2}\left(\int_{\R^N}|u^{k+1}|^{2^*_\gamma}\dd z\right)^{\frac{2}{2^*_\gamma}}
            \leq
            C
            \delta^{p-2}\left(\int_{\R^N} |u|^{(k+1)\frac{4N_\gamma}{2N_\gamma-\mu}} \dd z \right)^{\frac{2N_\gamma-\mu}{2N_\gamma}},
    \end{equation}
    therefore,
    \begin{equation}
        \|u\|_{(k+1)2^*_\gamma}^{2(k+1)} 
        \leq 
        \frac{1}{S_\gamma}C(\delta,p,\gamma)\frac{(k+1)^2}{(2k+1)} \|u\|_{(k+1)\frac{4N_\gamma}{2N_\gamma-\mu}}^{2(k+1)},
    \end{equation}
    that is
    \begin{equation}
    \label{eq_bootstrap}
        \|u\|_{(k+1)2^*_\gamma} 
        \leq 
        \frac{1}{S_\gamma^{2(k+1)}}\tilde{C}(\delta,p,\gamma)\frac{(k+1)^{\frac{1}{k+1}}}{(2k+1)^{\frac{1}{2(k+1)}}} \|u\|_{(k+1)\frac{4N_\gamma}{2N_\gamma-\mu}}.
    \end{equation}
    Set $r\coloneqq 4N_\gamma/(2N_\gamma-\mu)$. Note that $2^*_\gamma > r$ thanks to the condition $\mu<4$. From the latter inequality, we can now start with a boostrap argument: since $u\in L^{2^*_\gamma}(\R^N)$ we can apply \eqref{eq_bootstrap} with $k+1 = 2^*_\gamma/r$ to get that $u\in L^{(k+1)2^*_\gamma}(\R^N) = L^{(2^*_\gamma)^2/r}(\R^N)$. Repeating the same argument, we
    obtain that $u \in L^q(\R^N)$ for every $q \in [2^*_\gamma,\infty)$, hence $u \in L^q(\R^N)$ for every $q \in [2,\infty)$.
    \end{proof}

\begin{Lem} \label{lem_reg_2}
Let $u \in H^1_\gamma(\R^N)$ be a solution to \eqref{GC}. Then $u \in L^\infty(\R^N)$.
\end{Lem}

    \begin{proof}
Let $K(z)\coloneqq d(z)^{-\mu}\ast |u|^p$. We first prove that $K\in L^\infty(\R^N)$.
We write
\begin{equation}
\label{eq_bound_K}
K(z) = \int_{\R^N} \frac{|u(w)|^p}{d(z-w)^\mu}\,\dd w = \int_{d(z-w)<1} \frac{|u(w)|^p}{d(z-w)^\mu}\,\dd w + \int_{d(z-w)\ge1} \frac{|u(w)|^p}{d(z-w)^\mu}\,\dd w \coloneqq I_1 + I_2.  
\end{equation}
By H\"{o}lder's inequality
\begin{equation}
\label{eq_bound_K1}
    I_1 \le \left( \int_{d(z-w)<1} d(z-w)^{-\mu s} \dd w \right)^{\frac{1}{s}} \left( \int_{d(z-w)<1} |u(w)|^{ps'} \dd w\right)^{\frac{1}{s'}}
\end{equation}
which is bounded for every $z \in \mathbb{R}^N$ provided that $\mu s < N_\gamma$ and $s' p \ge 2$. We observe that these choices are feasible. Indeed, if $p \ge 2$, the only requirement is $s \in \left(1,N_\gamma/\mu\right)$. If instead $p \in (1,2)$, we may choose $
s \in \left(1,\; \min\left\{N_\gamma/\mu,\; 2/(2-p)\right\}\right).
$\\
Moreover
\begin{equation}
\label{eq_bound_K2}
    I_2 \le \int_{d(z-w)>1} |u(w)|^p\dd w \le \|u\|_{p}^p<+\infty.
\end{equation}
Combining \eqref{eq_bound_K1} and \eqref{eq_bound_K2} in \eqref{eq_bound_K} we get that $K\in L^\infty(\R^N)$.
So $u$ satisfies $-\Delta_\gamma u = \Psi(z,u)\coloneqq - u + K(z) |u|^{p-1}$ with
\begin{equation}
    |\Psi(z,u)|\le
    \begin{cases}
        C (1+|u|^{r-1})& \quad \hbox{if $p\in (1,2)$,~ for every $r>2$,}\\
        C (1+ |u|^{p-1})&\quad \hbox{if $p\ge 2$.}
    \end{cases}
\end{equation}
We proceed with the case $p\in (1,2)$, the other is analogous.
Define $\varrho \ge \max\{ 1,\|u\|_r^{-1} \}$ and set
$$
v\coloneqq \frac{u}{ \varrho \|u\|_r},\quad  w_k\coloneqq (v-1+2^{-k})^+,\, k\in \N\cup\{0\}.
$$
Note that $w_k\in H^1_\gamma(\R^N)$ and $0\le w_{k+1}\le w_k$ a.e. in $\R^N$. Let also $U_k\coloneqq \|w_k\|_r^r$
and finally, we define 
$$\Omega_k\coloneqq \{ z\,: w_k(z)>0 \} = \{z\,: v(z) >1 - 2^{-k} \}.$$
Observe that the following properties hold true.
\begin{enumerate}[label=(\roman*)]
    \item $w_k>2^{-(k+1)}$ in $\Omega_{k+1}$. In fact, $v> 1- 2^{-(k+1)}$ in $\Omega_{k+1}$,  and
    $$
    w_k = v - 1 + 2^{-k} \ge 1- 2^{-(k+1)} -1 +2^{-k} = 2^{-(k+1)}.
    $$
    \item $|\Omega_{k+1}|\le 2^{(k+1)r} U_k$. By (i) we have
    $$
    |\Omega_{k+1}| = \int_{\Omega_{k+1}}1\dd z \le \int_{\Omega_{k+1}} \left( \frac{w_k}{2^{-(k+1)}} \right)^r\dd z \le 2^{(k+1)r}\int_{\Omega_{k+1}}|w_k|^r\dd z \le U_k 2^{(k+1)r}.
    $$
    \item $w_{k+1}\le 2^{-(k+1)}$ in $\Omega_{k+1}\cap \{u\le 1\}$. In fact, if $u\le 1$ then $v\le u \le 1$ and
    $$
    w_{k+1}\le (1-1+2^{-(k+1)})^+ = 2^{-(k+1)}.
    $$
    \item $v<(2^{k+1}+1)w_{k}$ in $\Omega_{k+1}$. By (i) we have that $1<w_k 2^{k+1}$ in $\Omega_{k+1}$ and so
    $$
    v = w_k +1 - 2^{-k} < w_k(2^{k+1}+1)-2^{-k} < w_k(2^{k+1}+1).
    $$
\end{enumerate}
In the rest of the proof, $C$ denotes a generic constant independent of $k$, which may change from line to line.
Now we test the equation $-\Delta_\gamma u = \Psi(z,u)$ with $w_{k+1}$. Noting that, for all $k$, $\nabla_\gamma w_{k+1} = \nabla_\gamma v$ in $\Omega_{k+1}$ and $\nabla_\gamma w_{k+1}=0$ in $\R^N\setminus\Omega_{k+1}$ we get
\begin{align*}
    \| \nabla_\gamma w_{k+1}\|_2^2 &= \int_{\R^N} \nabla_\gamma v \cdot  \nabla_\gamma w_{k+1} = \left( \varrho \|u\|_r \right)^{-1} \int_{\R^N} \nabla_\gamma u \cdot \nabla_\gamma w_{k+1}\\
    & = \left( \varrho \|u\|_r \right)^{-1}\int_{\R^N} \Psi(z,u) w_{k+1}\dd z = \left( \varrho \|u\|_r \right)^{-1}\int_{\Omega_{k+1}} \Psi(z,u) w_{k+1}\dd z  \\
    & \le C  \left( \varrho \|u\|_r \right)^{-1} \left( \int_{ \Omega_{k+1}\cap \{u\le 1\} } w_{k+1}\dd z + \int_{ \Omega_{k+1}\cap \{u> 1\} } u^{r-1} w_{k+1}\dd z  \right)\\
      & \le C  \left( \varrho \|u\|_r \right)^{-1} \left( \int_{ \Omega_{k+1}\cap \{u\le 1\} } w_{k+1}\dd z + \int_{ \Omega_{k+1}\cap \{u> 1\} } u^{r-1} w_{k}\dd z  \right)
\end{align*}
by the growth condition of $\Psi$ and the inequality $w_{k+1}\le w_k$. \\
We now estimate the first integral by (iii) and (ii) as
\begin{equation}
    \label{est_int_1}
    \int_{ \Omega_{k+1}\cap \{u\le 1\} } w_{k+1}\dd z\le 2^{-(k+1)}|\Omega_{k+1}|\le 2^{(k+1)(r-1)} U_k\le \left(\varrho\|u\|_r \right)^{r-1}2^{(k+1)(r-1)} U_k
\end{equation}
since $\varrho\|u\|_r>1$ by definition of $\varrho$. We estimate the second integral by (iv) as
\begin{equation}
    \label{est_int_2}
    \begin{aligned}
        \int_{ \Omega_{k+1}\cap \{u> 1\} } u^{r-1} w_{k}\dd z &=\left(\varrho \|u\|_r \right)^{r-1} \int_{\Omega_{k+1}\cap \{u>1\}} v(z)^{r-1} w_k \\
        &
        \le \left(\varrho \|u\|_r \right)^{r-1}(2^{k+1}+1)^{r-1} U_k \\
        &\le C\left(\varrho \|u\|_r \right)^{r-1}(2^{k+1})^{r-1} U_k.
    \end{aligned}
\end{equation}
So combining \eqref{est_int_1} and \eqref{est_int_2} we get
\begin{equation}
    \label{eq_nablawk}
\|\nabla_\gamma w_{k+1}\|_2^2 \le C \left( \varrho\|u\|_r \right)^{r-2} 2^{(k+1)(r-1)} U_k.
\end{equation}
Now, by H\"{o}lder's and Sobolev inequality
\begin{equation}
    \label{eq:uk+1}
U_{k+1} = \int_{\Omega_{k+1}} |w_{k+1}|^r \le \|w_{k+1}\|_{2^*_\gamma}^r |\Omega_{k+1}|^{1-\frac{r}{2^*_\gamma}}\le C \|\nabla_\gamma w_{k+1}\|_2^{r}|\Omega_{k+1}|^{1-\frac{r}{2^*_\gamma}}.
\end{equation}
Note that, again by (ii) we have
\begin{equation}
    \label{est_uk}
    |\Omega_{k+1}|^{1-\frac{r}{2^*_\gamma}} = |\Omega_{k+1}|^{1-\frac{r}{2} + \frac{r}{N_\gamma}} \le 2^{(k+1)r \left( 1-r/2^*_\gamma\right)} U_k^{1-\frac{r}{2} + \frac{r}{N_\gamma}}
\end{equation}
and so combining \eqref{eq_nablawk} and \eqref{est_uk} in \eqref{eq:uk+1} 
$$
U_{k+1}\le C|\Omega_{k+1}|^{1-\frac{r}{2^*_\gamma}} \left(\left( \varrho\|u\|_r \right)^{r-2} 2^{(k+1)(r-1)}\right)^{\frac{r}{2}} = C^k \left(\varrho\|u\|_r \right)^{\frac{r^2}{2}-r} U_k^{1+\frac{r}{N_\gamma}}
$$
for some $C>1$ not depending on $k$. \\
Let $\eta\coloneqq C^{-\frac{N_\gamma}{r}}\in (0,1)$. Choosing
$$
\varrho\coloneqq \max\left\lbrace 1,\|u\|_r, \left( \|u\|_r^{r^2/2- r} \eta^{-1}\right)^{\frac{1}{\xi}}\right\rbrace,
$$
where $\xi\coloneqq \frac{r^2}{N_\gamma} + r -\frac{r^2}{2}>0$, since we can choose $r<2^*_\gamma$.\\
We now prove by induction that $U_k\le \frac{\eta^k}{\varrho^r}$ for all $k\in\N\cup\{0\}$. Clearly,
$$
U_0 = \|v^+\|_r^r \le \|v\|_r^r = \frac{1}{\varrho^r}.
$$
Suppose now $U_k\le \frac{\eta^k}{\varrho^r}$. We estimate
\begin{align}
U_{k+1} &\le C^k \left( \varrho \|u\|_r \right)^{ -r} \left( \frac{\eta^k}{\varrho^r} \right)^{1+\frac{r}{N_\gamma}} = \frac{\eta^k}{\varrho^r} \varrho^{-\left( \frac{r^2}{N_\gamma} + r - \right)} \|u\|_r^{r^2/2-2}
\\
&\le \frac{\eta^k}{\varrho^r} \left(\|u\|_r^{r^2/2 -r}\eta^{-1} \right)^{-1} \|u\|_r^{r^2/2-r}    = \frac{\eta^{k+1}}{\varrho^r}. 
\end{align}
So we have that $U_k\to 0$ as $k\to\infty$. Since $w_k\to (v-1)^+$ a.e. in $\R^N$ as $k\to\infty$ and $w_k\le v\in L^r(\R^N)$, it follows by the dominated convergence theorem that $\|(v-1)^+\|_r = 0$. So $v\le 1$ a.e. in $\R^N$ and then $\|u\|_\infty\le \varrho\|u\|_r$, concluding the proof. 
\end{proof}
We can now present the proof of Theorem \ref{thm_reg}.
\begin{proof}[Proof of Theorem \ref{thm_reg}]
    By Lemmas \ref{lem_reg_1} and \ref{lem_reg_2} we know that $u\in L^q(\R^N)$ for all $q\in [2,\infty]$. The local (H\"{o}lder-) continuity of the solution $u$ comes from the nonhomogeneous Harnack inequality for $X$-elliptic operators established in \cite[Theorem 5.5]{gutierrezlanconelli} 
	as already observed in \cite{KL12}, see also \cite{Lascialfari_2002}.
\end{proof}

\section{Poho\v zaev identity and nonexistence}\label{Sect:nex}
In this section, we first establish a Poho\v zaev identity for smooth solutions to \eqref{GC}. As a consequence, we obtain the nonexistence result stated in Theorem~\ref{thm_nonexistence}.

As anticipated in the introduction, to obtain the Poho\v zaev identity, we follow the classical argument in \cite{MVS2015}, where the idea is to test the equation with the function $z \cdot \nabla u$. However, we need to adapt the argument to the Grushin setting: instead of $(x,y)\cdot \nabla u$, we use $ (x,(1+\gamma)y)\cdot \nabla u$ in order to recover the correct scaling in the estimates. 

The proof of the Poho\v zaev identity relies on \cite[Appendix A]{YuYangTang}: however, the novelty here is the presence of the convolution term, which has been treated adapting the argument of \cite[Proposition 3.1]{MVS2013}. For the sake of completeness, we report the whole detailed proof.

\begin{Prop}
\label{prop_Pohozaev}
    Let $u \in H^1_\gamma(\R^N) \cap L^\frac{2pN_\gamma}{2N_\gamma-\mu}(\R^N)$ be a weak solution to \eqref{GC}. If $u\in C^2(\R^N)$ then
$$
    \frac{N_\gamma-2}{2}\int_{\R^N} |\nabla_\gamma u|^2 \dd z + \frac{N_\gamma}{2}\int_{\R^N} |u|^2 \dd z = \frac{2N_\gamma -\mu}{2p}\int_{\R^N} \left( d(z)^{-\mu}\ast |u|^p \right) |u|^p \dd z.
$$
\end{Prop}
\begin{proof}
Take $\varphi\in C^1_c(\R^N)$ such that $\varphi\equiv 1$ on $B_1$. Given $z = (x,y)\in \R^N$ we denote by $\hat z = (x, (1+\gamma)y)$ and we consider 
$$
    v_\lambda(z) \coloneqq \varphi(\lambda z)\hat z\cdot \nabla u (z), \quad \lambda>0,
$$
as a test function in \eqref{GC:weak:formulation}, getting
\begin{equation}
    \label{eq_test}
    0=\int_{\R^N}\nabla_\gamma u\cdot \nabla_\gamma v_\lambda\dd z + \int_{\R^N} u v_\lambda\dd z - \int_{\R^N} \left( d(z)^{-\mu}\ast |u|^p \right)|u|^{p-2}u v_\lambda\dd z \eqqcolon I_1 + I_2 - I_3.
\end{equation}
We now carefully evaluate each term. Note that we can write $\Delta_\gamma u = \divv(A_\gamma\nabla u)$, where $A_\gamma$ is the $N\times N$ matrix given by 
$$
    A_\gamma=
    \begin{pmatrix}
        I_{m\times m} & 0 \\
        0 & |x|^{2\gamma } I_{\ell\times\ell},
    \end{pmatrix}
$$
where $I$ is the identity matrix. Noting that $A_\gamma \nabla u\cdot \nabla v = \nabla_\gamma u\cdot\nabla_\gamma v $, by the divergence theorem we can write
$$
    \int_{\R^N}\Delta_\gamma u v_\lambda\dd z = \int_{\R^N}\divv(A_\gamma\nabla u) v_\lambda \dd z = -\int_{\R^N} A_\gamma\nabla u\cdot \nabla v_\lambda \dd z = -\int_{\R^N} \nabla_\gamma u\cdot \nabla_\gamma v_\lambda \dd z = - I_1.
$$
We get
\begin{align}
    -I_1 &= \int_{\R^N} v_\lambda \Delta_\gamma u \dd z = \int_{\R^N}\varphi(\lambda z) \left(\hat z\cdot \nabla u\right) \Delta_\gamma u\dd z \\&= \int_{\R^N} \varphi(\lambda z) \left(x\cdot \nabla_x u + (1+\gamma) y\cdot \nabla_y u  \right)\left( \Delta_x u + |x|^{2\gamma}\Delta_y u \right)\dd z \\ &=
    \int_{\R^N} \varphi(\lambda z) \left(x\cdot \nabla_x u \right)\Delta_x u\dd z+ \int_{\R^N} \varphi(\lambda z)|x|^{2\gamma}\left(x\cdot \nabla_x u \right)\Delta_y u\dd z \\ &\quad\quad + (1+\gamma) \int_{\R^N}\varphi(\lambda z) \left(y\cdot \nabla_y u \right)\Delta_x u \dd z + (1+\gamma) \int_{\R^N} \varphi(\lambda z) |x|^{2\gamma}\left(y\cdot \nabla_y u \right)\Delta_y u\dd z\\
    &\eqqcolon J_1 + J_2 + (1+\gamma)J_3 +(1+\gamma) J_4.
\end{align}
We now compute each term $J_i$ separately. The calculations are fairly standard; however, we report them for completeness and in order to make explicit how the Grushin terms appear.\\
For each term we follow essentially the same procedure: we first apply the divergence theorem, then use the dominated convergence theorem to remove the term involving $\varphi$. We then evaluate the remaining gradient term and finally apply the divergence theorem once more. To simplify notation, we indicate the limits $\lim_{\lambda\to 0}J_i$, just with $J_i$, for all $i=1,\dots 4$.\\
By the divergence theorem and the dominated convergence theorem we get
\begin{align}
    J_1 &=- \int_{\R^N}
    \nabla_x u \cdot \nabla_x (\varphi(\lambda z) (x \cdot \nabla_x u))\dd z  \\ 
    &=-\int_{\R^N}\lambda (x\cdot \nabla_x u) \nabla_x u\cdot \nabla_x\varphi(\lambda z)\dd z -\int_{\R^N}\varphi(\lambda z) \nabla_x u \cdot \nabla_x \left(x\cdot \nabla_x u \right)\dd z\\
    &\to -\int_{\R^N} \nabla_x u \cdot \nabla_x \left(x\cdot \nabla_x u \right)\dd z
\end{align}
as $\lambda\to 0$.
Noting that
\begin{equation}\label{est_grad_1}
    \nabla_x u \cdot \nabla_x \left(x\cdot \nabla_x u \right) = |\nabla_x u|^2 + \sum_{i,k = 1}^m x_i \frac{\partial^2 u}{\partial x_i \partial x_k} \frac{\partial u}{\partial x_k} = |\nabla_x u|^2 + \frac12 x\cdot \nabla_x \left( |\nabla_x u|^2 \right)
\end{equation}
again by the divergence theorem
$$
    \int_{\R^N} \nabla_x u \cdot \nabla_x \left(x\cdot \nabla_x u \right)\dd z = \int_{\R^N}|\nabla_x u|^2 \dd z -\frac12 \int_{\R^N}\divv_x(x)|\nabla_x u|^2\dd z =\left(1-\frac{m}{2}  \right)\int_{\R^N}|\nabla_x u|^2\dd z,
$$
leading to the estimate
\begin{equation}\label{eq_J1}
    J_1 = \left(\frac{m}{2}-1\right)\int_{\R^N}|\nabla_x u|^2 \dd z.
\end{equation}
Regarding $J_2$, again by the divergence theorem
\begin{align*}
    J_2 &= -\int_{\R^N}|x|^{2\gamma}  \nabla_y u \cdot \nabla_y \left( \varphi(\lambda z) (x\cdot \nabla_x u)\right)\dd z. 
    \\ & =\int_{\R^N}|x|^{2\gamma}  \nabla_y u \cdot \left[\lambda\nabla_y \varphi(\lambda z)(x\cdot \nabla_x u) + \varphi(\lambda z)\nabla_y(x\cdot \nabla_x u)\right] \dd z
\end{align*}
and, passing to the limit as $\lambda\to 0$ we get
\begin{equation}
    J_2 = -\int_{\R^N} |x|^{2\gamma} \nabla_y u\cdot \nabla_y (x\cdot \nabla_x u)\dd z.
\end{equation}
We compute
\begin{equation}\label{est_grad_2}
    \nabla_y u\cdot \nabla_y\left(x\cdot \nabla_x u \right) = \sum_{j=1}^\ell \frac{\partial u}{\partial y_j}\left(\sum_{i=1}^m x_i \frac{\partial^2 u}{\partial x_i \partial y_j} \right) = \frac12 \sum_{i=1}^m x_i \frac{\partial}{\partial x_i}\left(\sum_{j=1}^\ell \left\lvert \frac{\partial u}{\partial y_j} \right\rvert^2 \right)= \frac12 x\cdot \nabla_x (\left\lvert\nabla_y u \right\rvert^2),
\end{equation}
and so by the divergence theorem
$$
    J_2 = -\frac{1}{2}\int_{\R^N} |x|^{2\gamma} x\cdot \nabla_x (\left\lvert\nabla_y u \right\rvert^2)\dd z = \frac{1}{2}\int_{\R^N} |\nabla_y u|^2 \divv_x\left( |x|^{2\gamma} x \right)\dd z.
$$
Clearly $\divv_x\left( |x|^{2\gamma} x \right) = 2\gamma |x|^{2\gamma} + m |x|^{2\gamma}$, that is
\begin{equation}\label{eq_J2}
    J_2  = \left(\frac{m}{2} +\gamma \right)\int_{\R^N} |x|^{2\gamma} |\nabla_y u|^2 \dd z. 
\end{equation}
We now evaluate $J_3$.
\begin{align}
    J_3 &= -\int_{\R^N}\nabla_x u\cdot \left(\lambda\nabla\varphi(\lambda z) (y\cdot \nabla_y u) \right)\dd z - \int_{\R^N}\varphi(\lambda z) \nabla_x u \cdot \nabla_x(y\cdot \nabla_y u)\dd z  \\
    &\to -\int_{\R^N} \nabla_x u \cdot \left(\nabla_x (y\cdot \nabla_y u) \right)\dd z,
\end{align}
as $\lambda\to 0$.
Arguing as in \eqref{est_grad_2}, $\nabla_x u \cdot \nabla_x(y\cdot \nabla_y u)  = \frac12 y\cdot \nabla_y \left(|\nabla_x u|^2 \right)$, and so, by the divergence theorem,
\begin{equation}
    \label{eq_J3}
    J_3 =- \frac12 \int_{\R^N} y\cdot \nabla_y |\nabla_x u|^2\dd z =\frac12 \int_{\R^N} |\nabla_x u|^2 \divv_y(y) \dd z = \frac{\ell}{2}\int_{\R^N}|\nabla_x u|^2 \dd z.
\end{equation}
We can finally compute $J_4$. Reasoning as before, letting $\lambda\to 0$ and using \eqref{est_grad_1} in the $y$ variable, we have
\begin{align}
    J_4 
    &=
    -\int_{\R^N} |x|^{2\gamma} \nabla_y u\cdot \nabla_y (y\cdot \nabla_y u)\dd z 
    \\& 
    = -\int_{\R^N} |x|^{2\gamma}|\nabla_y u|^2 \dd z - \frac12 \int_{\R^N}|x|^{2\gamma} y\cdot \nabla_y \left(|\nabla_y u|^2\right) \dd z 
    \\&
    = -\int_{\R^N} |x|^{2\gamma}|\nabla_y u|^2 \dd z  + \frac12\int_{\R^N} |\nabla_y u|^2 \divv_y(|x|^{2\gamma}  y) 
    \\&
    =\left(\frac{\ell}{2} -1 \right) \int_{\R^N} |x|^{2\gamma}|\nabla_y u|^2 \dd z. \label{eq_J4}
\end{align}
Putting together \eqref{eq_J1}, \eqref{eq_J2}, \eqref{eq_J3}, and \eqref{eq_J4}, we get that
\begin{equation}
    \label{eq_I1}
    -I_1 = \frac12(N_\gamma-2)\left( \int_{\R^N}|\nabla_x u|^2 \dd z +\int_{\R^N}|x|^{2\gamma}|\nabla_y u|^2 \dd z \right) = \frac{N_\gamma-2}{2}\int_{\R^N}|\nabla_\gamma u|^2\dd z.
\end{equation}
Regarding the term $I_2$, by the divergence theorem,
\begin{equation}
    I_2 = \int_{\R^N} u(z) \varphi(\lambda z)\hat z\cdot \nabla u(z)\dd z = \int_{\R^N} \varphi(\lambda z) \hat z\cdot \nabla\left( \frac{|u|^2}{2} \right)\dd z = -\int_{\R^N} \frac{|u|^2}{2}\divv(\varphi(\lambda z) \hat z)\dd z.
\end{equation}
Now, since $\divv(\hat z) = N_\gamma$, we get
\begin{equation}\label{eq_div}
    \operatorname{\divv} (\varphi(\lambda z) \hat z) = N_\gamma \varphi(\lambda z ) + \lambda \nabla\varphi(\lambda z)\cdot\hat z    
\end{equation}
and so
$$
    I_2 = -\int_{\R^N} \left(N_\gamma \varphi(\lambda z) + \lambda\hat z \cdot \nabla \varphi(\lambda z)   \right) \frac{|u(z)|^2}{2}\dd z.
$$
Letting $\lambda\to 0$ by the dominated convergence theorem, we conclude that
\begin{equation}\label{eq_I2}
    I_2\to - \frac{N_\gamma}{2}\int_{\R^N}|u|^2 \dd z.
\end{equation}
We now evaluate the convolution term $I_3$.
\begin{align}
    I_3 
    &=
    \int_{\R^N}\left( d(z)^{-\mu}\ast |u|^p  \right)|u|^{p-2} u \varphi(\lambda z)\hat z\cdot \nabla u(z)\dd z
    \\&
   =\int_{\R^N\times \R^N} |u(w)|^p d(z-w)^{-\mu}\varphi(\lambda z) \hat{z}\cdot \nabla_z\left( \frac{|u|^p}{p}\right)(z) \dd z \dd w
   \\&
   = \frac{1}{2}\int_{\R^N\times \R^N} d(z-w)^{-\mu} |u(w)|^p \varphi(\lambda z) \hat z \cdot \nabla_z\left( \frac{|u|^p}{p}\right)(z)\dd z \dd w 
   \\&
   \quad \quad+ \frac{1}{2}\int_{\R^N\times \R^N} d(z-w)^{-\mu} |u(z)|^p \varphi(\lambda w) \hat w\cdot \nabla_w\left( \frac{|u|^p}{p}\right)(w) \dd z \dd w 
   \eqqcolon \frac12\cI_1 +\frac12 \cI_2.
\end{align}
By the divergence theorem
\begin{align}\label{eq_calI1}
    \cI_1 &= \int_{\R^N\times \R^N}d(z-w)^{-\mu} |u(w)|^p\varphi(\lambda z) \hat z\cdot \nabla_z\left(\frac{|u|^p}{p}\right)(z)\dd z \dd w \\&= -\int_{\R^N\times \R^N}\frac{|u(z)|^p}{p} |u(w)|^p\divv ( d(z-w)^{-\mu} \varphi(\lambda z) \hat z)\dd z \dd w   
\end{align}
since the divergence is with respect to the variable $z$. We calculate using \eqref{eq_div}
$$
    \divv \left( d(z-w)^{-\mu} \varphi(\lambda z) \hat z \right) = d(z-w)^{-\mu} \left( N_\gamma \varphi(\lambda z) + \lambda \nabla \varphi(\lambda z) \cdot \hat{z} \right) + \varphi(\lambda z) \nabla_z \left( d(z-w)^{-\mu} \right) \cdot \hat{z}.
$$
Combining this in \eqref{eq_calI1} we get
$$
    \begin{aligned}
        \cI_1
        = -\int_{\R^N\times \R^N} \frac{|u(z)|^p}{p}|u(w)|^p 
        \Big[ &\, d(z-w)^{-\mu} \left( N_\gamma \varphi(\lambda z)
              + \lambda \nabla \varphi(\lambda z) \cdot \hat{z} \right) \\
        &+ \varphi(\lambda z)\,
        \nabla_z \left(d(z-w)^{-\mu}\right)\cdot \hat z
        \Big] \, \dd z \, \dd w
    \end{aligned}
$$
and with the same calculations
$$
    \begin{aligned}
        \cI_2
        = -\int_{\R^N\times \R^N} \frac{|u(w)|^p}{p}|u(z)|^p
        \Big[ &\, d(z-w)^{-\mu} \left( N_\gamma \varphi(\lambda w)
              + \lambda \nabla \varphi(\lambda w) \cdot \hat{w} \right) \\
        &+ \varphi(\lambda w)\,
        \nabla_w \left(d(z-w)^{-\mu}\right)\cdot \hat w
        \Big] \, \dd z \, \dd w .
        \end{aligned}
$$
Now note that
$$
    \nabla_w \left( d(z-w)^{-\mu} \right) = -\mu d(z-w)^{-\mu-1}\nabla_w\left(d(z-w) \right) = \mu d(z-w)^{-\mu-1}\nabla_z\left(d(z-w) \right)
$$
and so
\begin{align}
    \cI_1 + \cI_2
    & = - \frac{2}{p} \int_{\mathbb{R}^N\times \mathbb{R}^N} |u(w)|^p |u(z)|^p d(z-w)^{-\mu} (N_\gamma \varphi(\lambda z) + \hat z\cdot \nabla \varphi(\lambda z) ) \dd z \dd w
    \\&
    + \frac{\mu}{p} \int_{\mathbb{R}^N\times \mathbb{R}^N} |u(w)|^p |u(z)|^p \left[ d(z-w)^{-\mu-1}\nabla_z(d(z-w) \cdot \left(   \hat z \varphi(\lambda z) - \hat w \varphi(\lambda w) \right)\right]\dd z \dd w.
\end{align}
Now we want to let $\lambda\to 0$. By a straightforward calculation we have
\begin{equation}
    \nabla_z (d(z-w))\cdot (\hat z - \hat w) = d(z-w)
\end{equation}
and then by the dominated convergence theorem
\begin{align}
 \cI_1+ \cI_2 \to 
 -\frac{2 N_\gamma}{p}\int_{\R^N\times \R^N} d(z-w)^{-\mu} |u(w)|^p |u(z)|^p +\frac{\mu}{p}\int_{\R^N\times \R^N} d(z-w)^{-\mu} |u(w)|^p |u(z)|^p,
\end{align}
as $\lambda \to 0$.\\
That is,
\begin{equation}
    \label{eq_I3}
    I_3 = \int_{\R^N}\left( d(z)^{-\mu}\ast |u|^p  \right)|u|^{p-2} u v_\lambda\dd z  \to -\frac{2N_\gamma-\mu}{2p} \int_{\R^N}\left( d(z)^{-\mu}\ast |u|^p\right) |u|^p\dd z.
\end{equation}
as $\lambda\to 0$.
Plugging \eqref{eq_I1}, \eqref{eq_I2}, and \eqref{eq_I3} in \eqref{eq_test} we can conclude.
\end{proof}
We can now prove the nonexistence result, showing that if $p$ doesn't satisfy the assumption \eqref{eq_cond_p}, then the only weak solution to \eqref{GC} is the trivial one. 
\begin{proof}[Proof of Theorem \ref{thm_nonexistence}]
Testing equation \eqref{GC:weak:sol:def} with $u$ we get 
$$
    \int_{\R^N}|\nabla_\gamma u|^2 \dd z +\int_{\R^N} |u|^2\dd z = \int_{\R^N} \left(  d(z)^{-\mu}\ast |u|^p\right) |u|^p\dd z,
$$
and, by Proposition \ref{prop_Pohozaev},   
$$
    \left( \frac{N_\gamma-2}{2} - \frac{2N_\gamma-\mu}{2p}  \right)\int_{\R^N}|\nabla_\gamma u|^2\dd z  +\left( \frac{N_\gamma}{2} - \frac{2N_\gamma-\mu}{2p} \right)\int_{\R^N} |u|^2\dd z =0.
$$
So, if $\frac{1}{p}\le \frac{N_\gamma-2}{2N_\gamma-\mu}$ or $\frac{1}{p}\ge \frac{N_\gamma}{2N_\gamma-\mu}$ we have $\int_{\R^N}|\nabla_\gamma u|^2\dd z   = \int_{\R^N} |u|^2\dd z = 0$, that is $u\equiv 0$. 
\end{proof}

\section*{Acknowledgments}
\noindent
The authors are members of \emph{Gruppo Nazionale per l’Analisi Matematica, la Probabilità e le loro Applicazioni} (GNAMPA) of the \emph{Istituto Nazionale di Alta Matematica} (INdAM) and are partially supported by INdAM-GNAMPA Project 2026 titled \textit{Structural degeneracy and criticality in (sub)elliptic PDEs} (CUP E53C25002010001).

\bibliography{Bibliography}
\bibliographystyle{abbrv}

\end{document}